\documentclass{amsart}
\usepackage{hyperref}
\newtheorem{theorem}{Theorem}[section]
\newtheorem{proposition}{Proposition}[section]
\newtheorem{lemma}[theorem]{Lemma}
\newtheorem{corollary}[theorem]{Corollary}

\theoremstyle{definition}

\newtheorem{example}[theorem]{Example}

\theoremstyle{remark}
\newtheorem{remark}[theorem]{Remark}
\numberwithin{equation}{section}

\newcommand{\cB}{{\mathcal B}}

\newcommand{\cD}{{\mathcal D}}
\newcommand{\bCE}{{\mathcal E}}

\newcommand{\bCG}{{\mathcal G}}
\newcommand{\cH}{{\mathcal H}}

\newcommand{\cK}{{\mathcal K}}

\newcommand{\cM}{{\mathcal M}}

\newcommand{\cN}{{\mathcal N}}
\newcommand{\bCO}{{\mathcal O}}

\newcommand{\bCS}{{\mathcal S}}

\newcommand{\cV}{{\mathcal V}}

\newcommand{\bT}{{\mathbb{T}}}

\newcommand{\bD}{{\mathbb{D}}}
\newcommand{\bC}{{\mathbb{C}}}
\newcommand{\bS}{{\mathbb{S}}}
\newcommand{\bR}{{\mathbb{R}}}
\newcommand{\bN}{{\mathbb{N}}}

\newcommand{\cBm}[1]{\left[\begin{smallmatrix} #1
		\end{smallmatrix}\right]}

\newcommand{\bTheta}{{\boldsymbol{\bTheta}}}

\usepackage{color, bm}
\begin{document}

\thanks{}

\title[Dilation]{Dynamic Nevanlinna-Pick Theory, Covariance Dilations, and Non-commutative Varieties}

\author{Sourav Ghosh}
\address{Department of Mathematics, Indian Institute of Science Education and Research Pune, Maharashtra 411008}
\email{sourav.ghosh@students.iiserpune.ac.in}
\author{Haripada Sau}
\address{Department of Mathematics, Indian Institute of Science Education and Research Pune, Maharashtra 411008}
\email{hsau@iiserpune.ac.in}
\thanks{The author Sau was supported by the Anusandhan National Research Foundation (ANRF), Government of India, under the Mathematical Research Impact-Centric Support (MATRICS) scheme (Grant No. ANRF/ARGM/2025/000573/MTR)}


\subjclass[2020]{Primary 47A20, 47A57; Secondary 47B33, 47A25, 30E05, 46E22.}
\keywords{Nevanlinna-Pick interpolation, Blaschke product, composition operator, operator dilation, commutant lifting, non-commutative variety, complete spectral set.}

\begin{abstract} 
We establish a dynamic generalization of Nevanlinna-Pick interpolation for functions invariant under the action of a finite Blaschke product $f$, reducing global bounded holomorphic extension on orbit spaces to structured block-kernel positivity. Furthermore, we demonstrate that membership in $H^\infty(\mathbb{D})$ is universally detectable via deformations by any finite Blaschke product. These results are proved via an underlying operator-theoretic lifting framework for covariant operator pairs in the spirit of Sarason. Finally, in the setting of non-commutative function theory, we show that despite the universal validity of the matrix-valued von Neumann inequality over the free polydisk, Arveson-type complete spectral set representations break down for non-commutative inner varieties.
\end{abstract}

\maketitle


\section{Introduction}
The Nevanlinna-Pick interpolation theorem stands as a foundational landmark in function theory and operator theory, characterizing bounded holomorphic functions on the unit disk $\bD$. Let $H^\infty(\bD)$ denote the Banach algebra of bounded analytic functions on $\bD$ equipped with the supremum norm $\|\cdot\|_\infty$. \textit{Given a subset $\bCS\subset\bD$, the classical Pick theory establishes that a function $g: \bCS \to \bC$ admits a bounded holomorphic extension to whole $\bD$ with $\|g\|_\infty \le M$ if and only if the Pick kernel 
\[
(z,w) \mapsto \frac{M^2 - g(z)\overline{g(w)}}{1-z\overline{w}}
\] 
is positive semi-definite on $\bCS \times \bCS$, where $k(z,w) := (1-z\overline{w})^{-1}$ is the Szeg\"o kernel.} While this classical paradigm relies fundamentally on the identity map $f(z)=z$, modern operator theory increasingly demands extension and interpolation criteria adapted to non-linear symmetries, fixed-point spaces of composition operators $C_f: g \mapsto g \circ f$, and the discrete iteration theory of inner functions.

For any non-identity holomorphic self-map $f$ of $\bD$, the classical Denjoy--Wolff theorem asserts the existence of a unique point $p \in \overline{\bD}$ toward which forward iterations $f^{\circ n}$ converge; see e.g., \cite[Theorem 2.51]{MR1237406}. Hereafter, we adopt the notation
$
f^{\circ n}
$ for $n$-fold composition of $f$ with the convention that $f^{\circ 0}$ is the identity map. Depending on the geometric nature of the Denjoy--Wolff point $p$--whether it lies in the interior $\bD$ (elliptic) or on the boundary $\partial\bD$ (parabolic or hyperbolic)--one can analyze the non-emptiness and structure of the fixed-point space $\ker(I - C_f)$ of the composition operator $C_f$; see e.g., \cite[Theorem 2.63]{MR1237406}. However, determining whether a given solution to $g \circ f = g$ possesses sufficient regularity to belong to a prescribed function space, such as $H^\infty(\bD)$, is a notoriously delicate problem; see \cite[Chapter 7]{MR1237406}.

This delicacy is further compounded when considering a \emph{partial solution}: a function $g: \bCS \to \bC$ defined merely on an $f$-invariant subset $\bCS \subset \bD$ without any presupposed holomorphic structure, satisfying $g(f(z)) = g(z)$ for all $z \in \bCS$. Our first main result settles this problem by establishing an exact criterion for when such a partial fixed point admits a global bounded holomorphic extension to all of $\bD$.

\subsection{A dynamic interpolation problem}
In this paper, we bridge classical Pick interpolation, operator theory, and holomorphic dynamics by introducing a dynamic Nevanlinna-Pick framework. We address the extension of partially defined functions that respect the orbit structure of a (finite) Blaschke product
$$
f(z)= e^{i\theta}\prod_{j=1}^N\frac{z-a_j}{1-\overline{a_j}z}.
$$
Given an invariant subset $\mathcal{S} \subset \mathbb{D}$ under the forward orbit
$$
\mathcal{O}_f(z):=\{f^{\circ n}(z):n\geq 0\},
$$ a function $g: \mathcal{S} \to \mathbb{C}$ satisfying the dynamical invariance $g(f(z)) = g(z)$ represents a partial eigenfunction of the composition operator 
$$
C_f:g\mapsto g\circ f.
$$ We establish that a function $g:\bCS\to \bC$ (holomorphic or not) admits a bounded holomorphic extension to all of $\mathbb{D}$ if and only if a structured matrix-valued function $\Delta: \mathcal{S} \times \mathcal{S} \to \mathcal{B}(\mathbb{C}^{n_z}, \mathbb{C}^{n_w})$--indexed by the orbit periodicities $n_z, n_w$--is positive semi-definite. This reduces a global $H^\infty(\mathbb{D})$ extension problem on quotient dynamics to a local block-positivity condition over periodic trajectories.
\medskip 

\noindent
\textbf{Theorem A.} [See Theorem \ref{T:BoundExt}] \textit{Let $\bCS\subset \bD$ and $g:\bCS\to\bC$ be a function for which there is a Blaschke function $f$ such that $\bCS\supset \{\mathcal{O}_f(z):z\in \bCS\}$ and 
$g(f(z)) = g(z)$ for all $z\in\bCS$. Then the following are equivalent:
\begin{enumerate}
    \item There is a bounded holomorphic function $G$ on $\bD$ such that $G|_S=g$;
    \item There exists $M>0$ such that the matrix-valued function $\Delta:(z,w)\to \cB(\bC^{n_z},\bC^{n_w})$ on $\bCS\times \bCS$ defined by
$$
\Delta(z,w) = \begin{bmatrix}
    \dfrac{M^2\alpha^2}{1-f^{\circ j}(z)\overline{f^{\circ l}(w)}}-\dfrac{g(z)\overline{g(w)}}{1-f^{\circ j+1}(z)\overline{f^{\circ l+1}(w)}}
\end{bmatrix}_{j,l=(0,0)}^{(n_z-1, n_w-1)}
$$ is a positive semi-definite kernel where 
$\alpha^2=(1+|f(0)|)(1-|f(0)|)^{-1}.$
\end{enumerate}
}\medskip

\noindent
Here, the matrix blocks explicitly encode orbit iterations scaled by the distortion factor $\alpha^2 = (1+\vert{}f(0)\vert{})(1-\vert{}f(0)\vert{})^{-1}$. 

Beyond orbit-invariant extension, our framework reveals a rather unexpected flexibility in classical function theory. While standard Nevanlinna-Pick theory rigidly fixes the identity map to test $H^\infty(\mathbb{D})$ membership, we show that $H^\infty(\mathbb{D})$ is universally detectable through the lens of \textit{any} finite Blaschke product $f$. More precisely, 
\medskip 

\noindent
\textbf{Theorem B.} [See Theorem \ref{T:Bdd-holomorphic}]
\textit{A given function $g:\bD\to\bC$ is bounded and holomorphic on $\bD$ if and only if there exists $M>0$ and a Blaschke function $f$ such that the function $\Delta_f:\bD\times \bD \to \bC$ defined by
$$
\Delta_f(z,w) =M^2\alpha^2k(z,w)-g(z)k(f(z),f(w))\overline{g(w)}
$$is a positive semi-definite kernel where $k(z,w) = (1-z\overline{w})^{-1}$ is the Szegö kernel and $\alpha^2 = (1+|f(0)|)(1-|f(0)|)^{-1}$.}

While the forward direction of Theorem \textbf{B} directly aligns with the classical Pick criterion with $f(z)=z$, the power of this result lies in its converse: positive semi-definiteness under an arbitrary finite Blaschke transformation $f$, balanced by the factor $\alpha^2$, is sufficient to imply global holomorphy and boundedness across the entire disk.

The following theorem plays a key role in obtaining the interpolation/extension Theorems \textbf{A} and \textbf{B} above. This may be viewed as an analogue of the classical result that states that the only bounded operator that commutes with the Hardy shift $T_z$ are the Toeplitz operators $M_\varphi$ with $\varphi\in H^\infty$.
\medskip

\noindent
\textbf{Theorem C.} [See Theorem \ref{T:f-commutant of vector-shift}] \textit{
Let $f$ be a Blaschke function and $A$ be a bounded operator on $H^2$ such that $f(M_z)A=AM_z$. Then $A=M_\varphi C_f$ for some $\varphi$ in $H^\infty$, the algebra of bounded analytic functions on $\bD$. Moreover, $\varphi$ can be chosen such that $$\|\varphi\|_{\infty}\le \|A\|\sqrt{\|f'\|_{\infty}}.$$}

\subsection{Lifting of operators satisfying covariance relations}
At the heart of our approach to proving Theorems \ref{T:BoundExt} and \ref{T:Bdd-holomorphic} is the operator-theoretic paradigm pioneered by D.~Sarason in his foundational work \cite{MR208383}. Sarason demonstrated that classical Nevanlinna-Pick interpolation can be reinterpreted through the algebra of compressed shift operators and the Commutant Lifting Theorem. Following Sarason's technique, our dynamic extension Theorem \textbf{A} and $H^\infty$ characterization Theorem \textbf{B} do not rely on purely function-theoretic estimates; rather, they are derived as concrete function-theoretic shadows of an underlying operator lifting and dilation theory.
 
Generalizing a large spectrum of dilation theorems by various authors including that of Sz.-Nagy \cite{MR236755} and And\^o \cite{MR155193}, we establish a lifting theorem for operator tuples satisfying an inner covariance relation of the form
\begin{align}\label{InnerCovariance}
T_1T_2=T_2f(T_1)
\end{align}
where $f$ is a Blaschke function. The intertwining relation \eqref{InnerCovariance} will be referred to as a covariance relation or an inner covariance relation when $f$ is an inner function. Given an operator pair $(X_1,X_2)$ acting on a Hilbert space $\cH$, an operator pair $(Y_1,Y_2)$ acting on $\cK\supset\cH$ is called a
\begin{enumerate}
    \item \textit{lift} of $(X_1,X_2)$ if $\cH$ is invariant under $(Y_1^*,Y_2^*)$ and $(Y_1^*,Y_2^*)|_\cH=(X_1^*,X_2^*)$;
    \item \textit{dilation} of $(X_1,X_2)$ if $\cH$ is semi-invariant under $(Y_1,Y_2)$ and $P_\cH(Y_1,Y_2)|_\cH=(X_1,X_2)$.
\end{enumerate}
We show that  
\medskip

\noindent
\textbf{Theorem D.} [See Theorem \ref{T:Lift1} and Theorem \ref{T:UniDil}] \textit{If $(T_1,T_2)$ acting on Hilbert space $\cH$ satisfies the inner covariance relation \eqref{InnerCovariance}, then 
\begin{enumerate}
    \item there is an isometric lift of $(T_1,T_2)$ satisfying the same inner covariance relation \eqref{InnerCovariance}; and also 
\item there is a unitary dilation of $(T_1,T_2)$ satisfying the same inner covariance relation. 
\end{enumerate}}
En route, we prove various lifting theorems that are interesting in their own right. 
\begin{remark}
    In the evolution of dilation theory, following Sz.-Nagy's seminal 1953 dilation theorem for a single contraction, And\^o (1963) famously extended the framework to commuting contractive pairs. Decades later, Sebesty\'en \cite{MR1176485} established that anti-commuting pairs ($T_1T_2=-T_2T_1$) likewise admit anti-commuting isometric dilations, a direction generalized in the 2019 paper \cite{MR3894905} to $q$-commuting pairs ($T_1T_2=qT_2T_1$ for $q\in\bT$). Upon completion of this manuscript, the authors became aware that a result equivalent to part (2) of Theorem \textbf{D} was independently established earlier in 2012 by Davidson and Katsoulis \cite{MR2929016}. We wish to explicitly acknowledge this overlap while underscoring the methodological distinctions between the two treatments. Whereas Davidson and Katsoulis approach covariant relations through $C^*$-algebraic and operator-algebraic machinery, our proof relies strictly on classical operator-theoretic techniques built around the Sz.-Nagy--Foias intertwining lifting theorem \cite{MR2760647}. This constructive perspective provides explicit operator representations that are essential for deriving the dynamic Nevanlinna-Pick criteria in Theorem \textbf{A} and the flexible test for $H^\infty(\bD)$ membership in Theorem \textbf{B}.
\end{remark}

\subsection{Failure of Arveson's dilation theorem in free analysis}
In classical operator theory, W.~Arveson's foundational work \cite{MR394232} established a profound equivalence between operator dilations and complete spectral sets as follows. Say that a compact set $K \subset \bC^d$ serves as a \textit{complete spectral set} for a commuting operator $d$-tuple $\bm T$, if the Taylor joint spectrum $\sigma_T(\bm T)\subset K$ and 
$$
\|f(T)\| \leq \sup_{z\in K}\|f(z)\|
$$for every matrix valued rational function $f$. When the inequality above holds for every scalar-valued rational functions, we say that $K$ is a spectral set for $\bm T$. Given a compact set $K$ in $\bC^d$, we say that a commuting $d$-tuple $\bm T$ of operators acting on $\cH$ admits a \textit{normal $K$-boundary dilation} $\bm N$ if $\bm N$ is a $d$-tuple of commuting normal operators acting on $\cK\supset\cH$ with $\sigma_T(\bm N)\subset bK$, the Shilov boundary of $K$ corresponding to the algebra of rational functions with poles away from $K$, and
$$
f(\bm T) = P_\cH f(\bm N)|_\cH
$$for every scalar-valued rational function $f$ with poles away from $K$.

Arveson showed that \textit{a commuting operator tuple $\bm T$ has a compact set $K$ in $\bC^d$ as a complete spectral set if and only if $\bm T$ admits a normal $K$-boundary dilation.}

In the modern framework of free analysis and non-commutative (NC) function theory, this correspondence between dilation and complete spectral set fails, as the following theorem states.
Indeed, while the free polydisk $\bD^{d,\rm NC}$, as defined below, acts as a universal complete spectral set for all contractive operator tuples, imposing specific non-commutative algebraic relations exhibits an unexpected phenomenon. For commutative operator pairs, algebraic constraints on the operators naturally restrict the spectrum and dilation domain to the associated algebraic variety; the reader is referred to the papers for work on this theme \cite{MR4693667, MR3584680, MR3834796}. However, when an operator pair $(T_1, T_2)$ satisfies a non-commutative inner covariance relation of the form $T_1 T_2 = T_2 f(T_1)$, the matricial structure fails to inherit this rigidity. More precisely,
\medskip

\noindent
\textbf{Theorem E.} [See Theorem \ref{T:CompleteSpec}] \textit{Every contractive operator tuple $$\bm T:= (T_1,T_2,\dots, T_d)$$ acting on $\cH$ for $d\geq1$, has the free polydisk
$$
\bD^{d,\rm NC}:=\bigcup_{n\geq1}\{(X_1,X_2,\dots,X_d)\in M_n(\bC)^d: \|X_j\|\leq 1 \mbox{ for each }j=1,2,\dots,d\}
$$as a complete spectral set, i.e., for every $N\times N$ matrix-valued NC polynomial $P$,
$$
\|P(\bm T)\|_{\cH\otimes\bC^N} \leq \sup_{\bm X\in\bD^{d,\rm NC}}\|P(\bm X)\|.
$$However, despite the lifting/dilation Theorem \textbf{C}, it is not necessary for a contractive pair $(T_1,T_2)$ satisfying an inner covariance relation $T_1T_2=T_2f(T_1)$ to have the following inner variety
$$
\cV_f^{\rm NC}:=\bigcup_{n\geq1}\{(X_1,X_2)\in M_n(\bC)^2: X_1X_2=X_2f(X_1) \mbox{ and } \|X_j\|\leq 1 \mbox{ for each }j\}
$$as a complete spectral set.}

In Example \ref{E:NotComplete}, we show that the non-commutative variety $\cV_f^{\rm NC}$ introduced above need not be a complete spectral set for a unitary pair satisfying the automorphic covariance relation $U_1U_2=e^{2\pi i \theta}U_2U_1$ where $\theta\in\bR\setminus\mathbb Q$.

\section{Operators satisfying covariance relations}\label{S:GenDiscussion}
We begin with a concrete example of matrix pair satisfying a covariance relation. 
\begin{example}
Let \(A,B\in M_n(\bC)\) and \(f:\mathbb{C}\to \mathbb{C}\) be such that \(f(A)\) is defined. A complex number \(\lambda\) is said to be a \textit{periodic point} of \(f\) if \(f^{\circ n}(\lambda)=\lambda\) for some \(n\in \bN\), and the smallest such \(n\) is called the \textit{period} of \(\lambda\). Let \(\lambda\) be a periodic point of \(f\) of period \(n\), then
\[A=\begin{bmatrix}
\lambda &  &  &  \\
 & f(\lambda) &  &  \\
 &  & \ddots &  \\
 &  & & f^{n-1}(\lambda)
\end{bmatrix}, \quad B = 
\begin{bmatrix}
 & 1 & & &  & \\
 & & 1&  & & \\
 &  &  &\ddots  & \\
 &  & & &1  \\
1 &  &  &  & 
\end{bmatrix} 
\]
satisfy \(f(A)B=BA\) and \(AB=Bf^{\circ (n-1)}(A)\). 
\end{example}
We now provide a large class of examples acting on reproducing kernel Hilbert spaces.
\begin{example}
Let $k$ be any reproducing kernel on $\bD$ where the multiplication operator $M_z:g\mapsto zg(z)$ is bounded. Suppose $k$ is also such that the composition operator $C_f:g\mapsto g\circ f$ is also bounded. Concrete Examples of such kernels include the Szeg\"o kernel and the Bergman kernel - see \cite[Chapter 3]{MR1397026}. It is straightforward to check that the covariance relation 
$$
f(T_z)C_f=C_fT_z
$$is always satisfied.
\end{example}
For a Hilbert space $\bCE$, a bounded operator that commutes with the unilateral shift $M_z$ on $H^2(\bCE)$ must be of the form $M_\varphi$ for a symbol $\varphi$ in $H^\infty(\cB(\bCE))$. This result plays an important role in operator theory, operator algebras, and interpolation theory. In this paper, we establish an analogue of this theorem by characterizing the structure of bounded operators that satisfy a left/right covariance relation with the shift operator. We shall find an application of this result later in this paper. 
\begin{theorem}\label{T:f-commutant of vector-shift}
Let $f$ be a Blaschke function and $A$ be a bounded operator on $H^2(\mathcal{E})$ such that 
$$
f(M_z)A=AM_z.
$$Then $A=M_{\varphi}C_f$ for some $\varphi\in H^{\infty}(\mathcal{B(E)})$ with $\|\varphi\|_\infty \leq \|A\|\sqrt{\|f'\|_{\infty,\bT}}$.
\end{theorem}
\begin{proof}
It is easy to see that $f(M_z)M_{\varphi}C_f=M_{\varphi}C_fM_z$ for any $\varphi \in H^{\infty}(\mathcal{B(E)})$. To prove the converse, observe that  $f(M_z)A=AM_z$ implies $T_{p\circ f}A=AT_p$ for all $p\in \bC[z]$. Define $\varphi:\bD\to \mathcal{B(E)}$ by 
$$
\varphi(z)e=A(1\otimes e)(z).
$$
Hence for every scalar-valued polynomial $p$,
\begin{align*}
    A(p\otimes e)(z)=p(f(z))\varphi(z)e .
\end{align*}By density of $\mathcal{E}$-valued polynomials in $H^2(\mathcal{E})$ and boundedness of $A$, the same identity extends to all $g\in H^2(\mathcal{E})$:
\begin{align}\label{mult-compo}
Ag(z)=\varphi(z)g( f(z))
\quad\mbox{for all } z\in\bD.
\end{align}
By the reproducing property of the kernel $k\otimes I_\bCE$, we have $$\langle \varphi(z)e_1,e_2\rangle=\langle A(1\otimes e_1)(z),e_2\rangle=\langle A(1\otimes e_1),k_z\otimes e_2\rangle.$$ This implies $\varphi(z)$ is bounded for all $z\in \bD$, and $\varphi$ is holomorphic. Once we prove that $\varphi$ is bounded, the conclusion follows immediately. For all $z,w\in\bD$, we compute 
\begin{align*}
k_z(f(w))\langle\varphi(w)e_1,e_2 \rangle _\bCE &= k_z(f(w))\langle e_1,\varphi(w)^*e_2 \rangle _\bCE\\
&=\left\langle k_z(f(w)) e_1,\varphi(w)^*e_2\right\rangle_\bCE\\
&=\left\langle \varphi(w) k_z(f(w)) e_1,e_2\right\rangle_\bCE\\
&=\langle A(k_z\otimes e_1)(w),e_2\rangle_{\bCE}=\langle A(k_z\otimes e_1),k_w\otimes e_2\rangle_{H^2(\bCE)}.    
\end{align*}
In the computation above, we used \eqref{mult-compo} for $g=k_z\otimes e_1$. We now take the modulus both sides to get 

\begin{align*}
    |k_z(f(w))||\langle \varphi(w)(e_1),e_2\rangle|=|\langle A(k_z\otimes e_1),k_w\otimes e_2\rangle|\le \|A\|\|k_z\|\|k_w\|\|e_1\|\|e_2\|.
\end{align*} Putting $z=\overline{f(w)}$ yields $$|\langle\varphi(w)e_1,e_2\rangle|\le \|A\| \sqrt{\dfrac{1-|f(w)|^2}{1-|w|^2}}\|e_1\|\|e_2\|.$$
Let us consider $h: \bD\to [0,\infty)$ defined by
$$
h(w)= \dfrac{1-|f(w)|^2}{1-|w|^2}.
$$ Since $f$ is a Blaschke function, $h$ is subharmonic on $\bD$. Indeed, let $f(w)=e^{i\theta}\prod_{j=1}^{N}\frac{w - a_k}{1 - \overline{a_k} w}$. Then an easy computation shows $$h(w)=\sum_{j=1}^{N}|\psi_j(w)|^2,$$ where $\psi_j(w) = \frac{\sqrt{1-
|a_j|^2}}{1 - \overline{a_j} w} \prod_{k=1}^{j-1} \left( \frac{w - a_k}{1 - \overline{a_k}w} \right)$. Since $h$ is a finite sum of squared absolute values of holomorphic functions, $h$ is subharmonic. Also since $f$ is a Blaschke function,
$$
\lim_{w\to \zeta}h(w)=|f'(\xi)|
$$for every $|\zeta|=1$. For a proof of this fact, see \cite[Theorem 5.2.5]{MR3793610}). Then by maximum principle for subharmonic functions, $\|h\|_\infty=\|f'\|_\infty$. Hence $$|\langle \varphi(w)e_1,e_2\rangle|\le \|A\|\sqrt{\|f'\|_{\infty}}\|e_1\|\|e_2\|.$$ It then implies that for all $w\in\bD$, 
$$
\|\varphi(w)\|=\sup_{\|e_1\|=\|e_2\|=1}|\langle \varphi(w)e_1,e_2\rangle|\le \|A\|\sqrt{\|f'\|_\infty}.
$$ Hence we conclude from \ref{mult-compo} that $A=M_\varphi C_f$.
\end{proof}
\begin{remark}
An alert reader may have noticed from the proof of Theorem \ref{T:f-commutant of vector-shift} that the only use of Blaschke functions is that for such functions, it is guaranteed that
$$
M:=\sup_\bD\frac{1-|f(w)|}{1-|w|}<\infty.
$$It turns out that the only self maps for which this supremum is finite are the Blaschke functions. Indeed, the inequalities $0< 1-|f(w)|\leq M(1-|w|)$ on $\bD$ give
$$
\lim_{|w|\to 1-} |f(w)|=1.
$$It is known (see e.g., \cite[Theorem 3.5.2]{MR3793610}) that the only self maps that satisfies this are the Blaschke functions.

Therefore the proof does not work for functions other than the Blaschke functions. It is natural to wonder to what extent Theorem \ref{T:f-commutant of vector-shift} remains true. We note the somewhat surprising fact that the assertion of Theorem \ref{T:f-commutant of vector-shift} is not necessarily true for any holomorphic self map $f$. For example, let $f(z)=z/2$. Then for any $g\in H^2$ and $|w|\le \frac{1}{2}$, 
$$
|g(w)|=|\langle g,k_w\rangle|\le \frac{\|g\|_{H^2}}{1-|w|^2}\le \frac{4}{3}\|g\|_{H^2}.
$$ Hence the operator $A$ on $H^2$ defined by $$A(g)=\varphi(g\circ f),$$ where $\varphi(z)=\sum_{n=1}^{\infty}\frac{z^n}{n}$, is well-defined. Since $\varphi\in H^2$ and $g\circ f\in H^\infty$ for any $g\in H^2$, it follows that $A$ is well defined. The operator $A$ is bounded as well:
$$\|A(g)\|=\frac{1}{2\pi}\int_{0}^{2\pi}\left|\varphi(e^{i\theta})g\left(\frac{e^{i\theta}}{2}\right)\right|^2d\theta \le \frac{16}{9}\|g\|^2_{H^2}\|\varphi\|_{H^2}.
$$
It is an easy check that $A$ satisfies the left covariance relation $M_fA=AM_z$, yet it cannot be of the form $M_\psi C_f$ for any $\psi\in H^{\infty}$ as any such $\psi$ would satisfy $\psi=A(1)=\varphi\notin H^{\infty}$.
\end{remark}
The analogue of Theorem \ref{T:f-commutant of vector-shift} for right covariance relation is the following. 
\begin{theorem}\label{T:f-commutant of shift'}
 Let $f$ be any holomorphic self map of the unit disk and $A$ be a bounded operator on $H^2(\mathcal{E})$ such that $$
 M_zA=Af(M_z).
 $$Then $AC_f=M_{\varphi}$ for some $\varphi\in H^\infty(\cB(\mathcal E))$. 
\end{theorem}
\begin{proof}
    It can be seen by means of an easy computation that for every holomorphic self map $f$ of the unit disk $\bD$, $f(M_z)C_f=C_fM_z$ on $H^2(\mathcal E)$. Thus by multiplying $C_f$ on the right hand side of the given intertwining $M_z A=Af(M_z)$ we see that $AC_f$ commutes with $M_z$ and hence the conclusion follows.
\end{proof}
\begin{remark}
Unlike Theorem~\ref{T:f-commutant of vector-shift} for left covariance relations, Theorem~\ref{T:f-commutant of shift'} is somewhat limited in applicability because the intertwining operator is uniquely determined only on $\operatorname{Ran} C_f$ rather than the entire space. However, if $C_f$ admits a right inverse $R_f$, then $A = M_\varphi R_f$. Note that for $C_f$ to admit a right inverse, it must be surjective. Since the composition operator $C_f$ induced by a non-constant holomorphic self-map is always injective (an application of the Open Mapping Theorem), surjectivity implies that $C_f$ is fully invertible. By \cite[Theorem~5.2.1]{MR2270722}, $C_f$ is invertible if and only if $f$ is an automorphism of $\bD$. Consequently, when $f$ is not an automorphism, a complete characterization of operators satisfying the intertwining relation $M_z A = A M_f$ remains an open problem.
\end{remark}

\section{Lifting of operators satisfying a covariance relation}\label{S:Lifting}

In this section we prove a lifting theorem for operator pairs satisfying an inner covariance relation. This lifting theorem will be used in later sections to obtain the dilation theorem.

A couple of necessary terminologies are in order. We say that a unitary dilation $U$ on $\cK$ of a contraction $T$ acting on $\cH$ is \textit{minimal} if $\cK$ is the smallest reducing subspace for $U$ that contains $\cH$, i.e., $\cK=\bigvee_{n\in \mathbb Z}U^n\cH$. Once a minimal isometric lift $V$ of $T$ is obtained, a minimal unitary dilation $U$ of $T$ is the adjoint of a minimal isometric lift of $V^*$. The unitary operator $U$ is also the same as a minimal unitary extension of $V$. Furthermore, any two minimal isometric lifts or unitary dilations of a contractive operator are unique in an appropriate sense. Proofs of these facts are routine and can be found for example in \cite[Chapter I]{MR2760647}.

We now briefly recall the classical lifting theorem due to Sz.-Nagy--Foias \cite{MR2760647} and Douglas-Muhly-Pearcy \cite{MR236752}.
\begin{theorem}
\label{T:Classical_Iso_ILT}
 Let $T_1$ on $\cH_1$, $T_2$ on $\cH_2$ and $X:\cH_1\to\cH_2$ be contractions such that 
    $$
    T_2X=XT_1.
    $$If $V_1$ on $\mathcal K_{1+}$, and $V_2$ on $\mathcal{K}_{2+}$ are minimal isometric lifts of $T_1$ and $T_2$, respectively, then there is a bounded operator $Y_+:\cK_{1+}\to\cK_{2+}$ such that
    \begin{align}\label{ClassicalILT}
    V_2Y_+=Y_+V_1, \quad  \|Y_+\|=\|X\| \quad\mbox{and}\quad
    Y_+^*|_{\cH_2}=X^*.
    \end{align}
    Moreover, the conclusion remains true if the minimality assumption on the isometric lifts is dropped.
\end{theorem} 
\begin{proof}
We only prove the moreover part as the rest is Theorem II.2.3 in \cite{MR2760647}. To this end, let $V_j$ acting on $\cK_{j+}$ be any isometric lifts of $T_j$. Consider the spaces
    $$
    \cK_{j0}:=\bigvee_{m\geq 0}V_j^m\cH_j.
    $$Since $V_j^*|_{\cH_j}=T_j^*$ one can check that $\cK_{j0}$ is reducing under $V_j$ and 
    $$
    V_{j0}:=V_j|_{\cK_{j0}}
    $$are minimal isometric lifts of $T_j$. By Theorem II.2.3 in \cite{MR2760647}, there is an operator $Y_{+0}:\cK_{10}\to\cK_{20}$ such that
    $$
    V_{2+}Y_{+0}=Y_{+0}V_{1+}, \quad  \|Y_{+0}\|=\|X\| \quad\mbox{and}\quad
    Y_{+0}^*|_{\cH_2}=X^*.
    $$It then follows that
    $$
    Y_+=\begin{bmatrix}
        Y_{+0} &0\\
        0&0
    \end{bmatrix}:
    \begin{bmatrix}
        \cK_{10} \\
        \cK_{1+}\ominus \cK_{10}
    \end{bmatrix}\to    \begin{bmatrix}
        \cK_{20} \\
        \cK_{2+}\ominus \cK_{20}
    \end{bmatrix}
    $$is the desired operator satisfying \eqref{ClassicalILT}. It should be noted that, a priori, only one of $V_1,V_2$ could be minimal: if $V_1$ is minimal, then we choose $Y_+=\cBm{Y_+ \\ 0}$ and if $V_2$ is minimal, then we choose $Y_+=\cBm{Y_{+0} & 0}$.
\end{proof}
We now state and prove the lifting result for operator pairs satisfying an inner covariance relation.
\begin{theorem}\label{T:LiftEndo}
Let $T$ be an operator on $\cH$, $f$ be an inner function, and $X\in\cB(\cH)$ be such that
\begin{align}\label{CovarianceEndo}
TX=Xf(T).
\end{align}Then the following holds: 
\begin{enumerate}
    \item If $V$ acting on $\cK_+\supset\cH$ be any isometric lift of $T$, then there is $Y_+\in\cB(\cK_+)$ such that $\cH$ is co-invariant under $Y_+$ (and $V$), and 
    \begin{align}\label{SpecialILT}
VY_+=Y_+f(V),\quad \|Y_+\|=\|X\| \quad\mbox{and}\quad Y_+^*|_\cH = X^*.
\end{align}

\item If $U$ acting on $\cK\supset\cH$ be any unitary dilation of $T$, then there is $Y\in\cB(\cK)$ such that $\cH$ is semi-invariant under $Y$ (and $U$ as well), and
\begin{enumerate}
    \item $UY=Yf(U)$, $\|Y\|=\|X\|$ and $P_\cH(U,Y)|_\cH=(T,X)$
    \item The subspace 
    $$
    \cK_+':=\bigvee_{m\geq 0} U^m\cH
    $$is invariant under $Y$, and with
    $
    (V_+',Y_+'):=(U,Y)|_{\cK_+'}$
    we have
    \begin{align}\label{SpecialILT'}
    V_+'Y_+=Y_+f(V_+'),\quad \|Y_+'\|=\|X\| \quad\mbox{and}\quad Y_+^{'*}|_\cH=X^*.
    \end{align}
\end{enumerate}
Furthermore, if the right-covariance relation \eqref{CovarianceEndo} is replaced by the left-covariance relation
\begin{align}\label{CovarianceEndo'}
f(T)X=XT
\end{align}then the conclusion remains true if we make the same replacements in items (1), (2a) and (2b).
\end{enumerate}
\end{theorem}
\begin{proof}[Proof of item (1)] 
Let $V$ acting on $\cK_+\supset\cH$ be an isometric lift of the contractive operator $T$. Then since $f$ is inner, $f(V)$ is an isometric lift of $f(T)$. Thus we may apply Theorem \ref{T:Classical_Iso_ILT} to the pair
    $$
    (T_1,T_2) = (f(T),T)
    $$to obtain a bounded operator $Y_+\in \cB(\cK_+)$ such that $Y_+^*|_\cH=X^*$ and 
    $$
    VY_+=Y_+f(V).
    $$With this, the proof of item (1) is complete.
\end{proof}
\begin{proof}[Proof of item (2)]
Let $U$ acting on $\cK$ be any unitary dilation of $T$. With the $U$-invariant subspace $\cK_+$ as in item (2b), let us consider the isometry $V=U|_{\cK_+}$. Then $V$ is a minimal isometric lift of $T$, and $f(V)=f(U)|_{\cK_+}$ on $\cK_+$ is an isometric lift of $f(T)$. By Theorem \ref{T:Classical_Iso_ILT}, there is a bounded operator $Y_+$ on $\cK_+$ such that
$$
VY_+=Y_+f(V), \quad \|Y_+\|=\|X\|,\quad\mbox{and}\quad
Y_+^*|_\cH=X^*.
$$Taking the adjoint of the intertwining relation above we get
$$
Y_+^*V^*=f(V)^*Y_+^*=\widetilde{f}(V^*)Y_+^*,
$$where $\widetilde{f}$ is the reflection of $f$ defined as 
\begin{align}\label{reflection}
\widetilde{f}(z) = \overline{f(\overline{z})}.
\end{align} 
We now view $U^*$ is an isometric lift of $V^*$ and apply Theorem \ref{T:Classical_Iso_ILT} again to the left covariance relation above to obtain a bounded operator $Y$ on $\cK$ such that 
$$
Y^*U^*=\widetilde{f}(U^*)Y^*, \quad \|Y^*\|=\|Y_+^*\|=\|X^*\|\quad\mbox{and}\quad
Y|_{\cK_+}=Y_+:\cK_+\to\cK_+.
$$The adjoint of the left covariance relation is the same as item (2a). Thus we have proved item (2)

Finally, let us note from the construction in the proof of item (2) above that both $U$ and $Y$ have upper triangular matrix forms 
\begin{align*}
    (U,Y)=\left(\begin{bmatrix}
        V&*\\
        0&*
    \end{bmatrix},
    \begin{bmatrix}
        Y_+ & *\\
        0 & *
    \end{bmatrix}\right):\begin{bmatrix}
        \cK_+\\
        \cK\ominus \cK_+
    \end{bmatrix}\to\begin{bmatrix}
        \cK_+\\
        \cK\ominus \cK_+
    \end{bmatrix}
\end{align*}and the operators $V$, $Y_+$ are lower triangular 
  \begin{align*}
  (V,Y)=\left( 
  \begin{bmatrix}
        T&0\\
        *&*
    \end{bmatrix}
  ,\begin{bmatrix}
        X & 0\\
        * & *
    \end{bmatrix}\right):\begin{bmatrix}
        \cH\\
        \cK_+\ominus \cH 
    \end{bmatrix}\to\begin{bmatrix}
        \cH\\
        \cK_+\ominus \cH 
    \end{bmatrix}.
\end{align*}
Putting these triangular forms together we conclude that $\cH$ is semi invariant under $(U,Y)$.
The proof for the left covariance relation \eqref{CovarianceEndo} proceeds along the same line as above and so we omit the details.
\end{proof}

\section{Isometric lifting theorem}
In this section we prove an isometric dilation theorem for contractive pairs satisfying a left or right inner covariance relation. We first develop the required preliminary results which are interesting in their own rights.

\begin{lemma}\label{L:f-consequences}
    Let $T$ be a contraction, $f$ in $A(\bD)$ be any function and $X$ be any bounded operator such that
    $$
    TX=Xf(T).
    $$Then for any function $g\in A(\bD)$ and $n\geq 0$, we have
    $$
    g(T)X^n = X^n g(f^{\circ n}(T)).
    $$In particular, if $T$ is an isometry and $f,g$ are inner functions, then for every $n\geq0$,
    $$
    X^n = g(T)^*X^n g(f^{\circ n}(T)).
    $$
\end{lemma}
\begin{proof}
    A plain use of mathematical induction gives the following intertwining relations for every $n\geq 0$:
    \begin{align}\label{Iteration}
        T^n X=X f(T)^n \quad\mbox{and}\quad TX^n=X^nf^{\circ n}(T).
    \end{align}
    Applying this fact to the second equation in \eqref{Iteration}, we see that for every $r\geq 0$,
    $$
    T^rX^n=X^n (f^{\circ n}(T))^r.
    $$Consequently, for any function $g\in A(\bD)$ and $n\geq 0$, we have
    $$
    g(T)X^n = X^n g(f^{\circ n}(T)).
    $$If $T$ is an isometry and $f,g$ are inner functions, then  $g(f^{\circ n}(T))$ is an isometry for every $n\geq0$. This was to be proved.
\end{proof}
We now establish a couple of preliminary results for operator pairs satisfying right inner covariance relations. The next two propositions formulate the core technical components required for the proof of the main theorem of this section.
\begin{proposition}\label{P:BabyStep}
    Let $X,T$ be contractive operators acting on $\cH$ such that
    $$
    T^*T=I\quad\mbox{and}\quad TX=Xf(T)
    $$for an inner function $f$. Then the operator $V:H^2(\cD_X)\to H^2(\cD_X)$ defined by
    $$
    V(z^n\otimes D_X h) = z^n \otimes D_X f^{\circ(n+1)}(T)h
    $$for every $n\geq 0$ and $h\in\cH$ has the following properties
    $$
    V^*V=I\quad\mbox{and}\quad V M_z = M_z f(V).
    $$
\end{proposition}
\begin{proof}
    Define $C$ on $\operatorname{Ran}D_X$ by
    \begin{align}\label{C}
    CD_Xh=D_Xf(T)h \quad\mbox{for every }h\in \cH.
    \end{align} Since $T$ is an isometry and $f$ is inner, $f(T)$ is an isometry. The computation
    \begin{align*}
        \|D_Xf(T)h\|^2 = \langle f(T)h, f(T)h\rangle - \langle Xf(T)h, Xf(T)h \rangle &=
        \langle h, h\rangle - \langle TXh, TXh \rangle\\
        &=\|h\|^2-\|Xh\|^2=\|D_Xh\|^2
    \end{align*}shows that $C$ is an isometry on $\operatorname{Ran}D_X$. We extend it continuously to $\cD_X$ and so it remains an isometry on $\cD_X$. By an inductive argument, we deduce from \eqref{C} that for every $n\geq 0$ and $h\in\cH$,
    $$
    C^nD_Xh=D_X f(T)^nh.
    $$Since polynomials are dense in $A(\bD)$, by a limiting argument we get
    $$
    g(C)D_X=D_Xg(f(T))
    $$for every $g\in A(\bD)$. Apply this for $g=f^{\circ n}$ for every $n\geq 0$ to see that the operator $V$ as stated in the statement is actually the diagonal operator 
    \begin{align}\label{V}
    V=\operatorname{diag}( f^{\circ n}(C))_{n=0}^\infty: \ell^2(\cD_X)\to\ell^2(\cD_X),
    \end{align}and so $f(V)=\operatorname{diag}( f^{\circ (n+1)}(C))_{n=0}^\infty$. Since $C$ is an isometry, and $f$ is inner, each diagonal entry of $V$ is isometry and hence so is $V$. What remains is to check the intertwining $VM_z=M_zf(V)$. It is enough to check on the basis elements $z^n\otimes \xi$ for $n\geq 0$ and $\xi\in\cD_X$:
    \begin{align*}
        VM_z (z^n\otimes \xi) &= z^{n+1}\otimes f^{\circ{n+1}}(C)\xi\\
        &=(M_z\otimes I_{\cD_X})(z^n\otimes f^{\circ{n+1}}(C)\xi)=(M_z\otimes I_{\cD_X})f(V)(z^n\otimes \xi).
    \end{align*}This completes the proof of the assertion.
\end{proof}
In Theorem \ref{T:LiftEndo}, we constructed isometric lifts (and unitary dilations) for a contractive operator $T$, lifting the intertwining operator $X$ to the dilation space preserving its norm. The result below adopts a complementary perspective: given an isometry operator, we first dilate the intertwining operator $X$ and subsequently extend the isometry to the resulting dilation space preserving the isometric structure. This construction will serve as a crucial technical ingredient in establishing our main dilation theorems.
\begin{proposition}\label{P:FinalNail}
Let $f$ be an inner function, $V$ on $\cH$ be an isometry and $X$ on $\cH$ be a contraction such that
    $
    VX=Xf(V).
    $ If $Y$ on $\cK$ is a minimal isometric lift of $X$, then there is an isometric operator $S$ on $\cK$ such that $\cH$ reduces $S$ and
    \begin{align}\label{FinalNail}
    SY=Yf(S).
    \end{align}
\end{proposition}
We present two distinct proofs of this result. The first is primarily analytic, whereas the second leverages the geometric structure of classical dilation theory. Translating the problem into operator-theoretic terms with $T=\pi(z)$, our objective is to construct an isometry $S$ that commutes with $Y$ via the relation $SY=Yf(S)$, while leaving $\cH$ reducing. The proofs yield precisely such an operator.
\begin{proof}[The first proof]
    The minimality assumption on $Y$ gives
    $
    \cK= \bigvee_{n\geq 0} Y^n\cH.
    $ Let us set
    \begin{align}\label{DefineS}
    S (Y^nh):= Y^n f^{\circ n}(T)h \quad \mbox{for every }h\in\cH \mbox{ and } n\geq 0
    \end{align}
    and extend linearly to 
    $$
    \cM:=\operatorname{span}\{Y^nh:h\in\cH, n\geq 0\}.
    $$ We claim that $S$ on $\cM$ is an isometry and it satisfies the desired properties \eqref{FinalNail}. Once this is established on $\cM$, $S$ can be extended to $\cK$ continuously and the properties will continue to hold in $\cK$ by continuity of the operators involved.

    For $h,h'\in\cH$ and non-negative integers $m,r$ we compute
    \begin{align*}
        \langle S^{m+r}(Y^{m+r}h), S^r(Y^rh')\rangle &=  \langle Y^{m+r} f^{\circ(m+r)}(T)h, Y^rf^{\circ r}(T)h'\rangle_\cK \\
        &= \langle Y^{m} f^{\circ(m+r)}(T)h, f^{\circ r}(T)h'\rangle_\cK \\
        &= \langle X^{m} f^{\circ(m+r)}(T)h, f^{\circ r}(T)h'\rangle_\cH\\
        &=\langle f^{\circ r}(T)^*X^{m} f^{\circ(m+r)}(T)h, h'\rangle_\cH  \\
        &=\langle X^{m} h, h'\rangle_\cH
        =\langle Y^{m} h, h'\rangle_\cK=\langle Y^{m+r}h, Y^rh'\rangle_\cK.
    \end{align*}Here we used several times the facts that $Y^*|_\cH=X^*$,  $f^{\circ r}(T)$ is an isometry and Lemma \ref{L:f-consequences}. This shows that $S$ on $\cM$ is an isometry.

    We now turn to showing that $S$ has the properties as stated in \eqref{FinalNail}. For the intertwining relation, we proceed as follows.  We use the definition of $S$ and mathematical induction to note that for every $m,n\geq 0$ and $h\in \cH$,
    \begin{align}\label{DefineS'}
        S^mY^nh = Y^n (f^{\circ n}(T))^mh.
    \end{align}
Let the power series representation of $f$ around the origin be
    $
    f(z) = \sum_{m\geq0} \alpha_m z^m$.
    Use \eqref{DefineS'} to note that for every $n\geq 0$ and $h\in\cH$,
    \begin{align*}
        f(S)(Y^nh) = \sum_m \alpha_m S^mY^nh = \sum_m \alpha_m Y^n (f^{\circ n}(T))^mh = Y^n f^{\circ (n+1)}(T)h.
    \end{align*}
Finally, we use the above identity on $\cH$ to note that for every $n\geq 0$ and $h\in\cH$,
    \begin{align*}
        SY(Y^n h) =Y^{n+1} f^{\circ(n+1)}(T)(h)= Y Y^n f^{\circ(n+1)}(T)(h)=Yf(S)(Y^nh).
    \end{align*}

To see that $\cH$ reduces $S$ and $S|_\cH=T$, we first note by putting $n=0$ in \eqref{DefineS} that the restriction of $S$ to $\cH$ is $T$. To see that $S^*|_\cH=T^*$, we compute for every $n\geq0$, and $h\in \cH$
\begin{align*}
    \langle S^*h, Y^nh' \rangle = \langle h, SY^nh' \rangle &= 
    \langle h, Y^nf^{\circ n}(T)h' \rangle \\
    &= \langle h, X^nf^{\circ n}(T)h' \rangle \\
    &= \langle h, TX^nh' \rangle =\langle T^*h , X^nh'\rangle = \langle T^*h , Y^nh'\rangle.
\end{align*}
    In the computation above we again made crucial use of Lemma \ref{L:f-consequences} and the fact that $Y$ is a lift of $X$. This completes the first proof of the proposition.
\end{proof}
For the second proof we use an elegant construction of isometric lift due to Sch\"affer \cite{MR68740}: Given a contractive operator $T$ on $\cH$, it is an easy check that operator matrix
\begin{align}\label{Schaffer}
    V=\begin{bmatrix}
        T&\\
        D_T& \\
        & I_{\cD_T} &\\
        & & I_{\cD_T}\\
        & & & \ddots
    \end{bmatrix}:\begin{bmatrix}
        \cH \\ \cD_T \\ \cD_T \\ \vdots
    \end{bmatrix}\to \begin{bmatrix}
        \cH \\ \cD_T \\ \cD_T \\ \vdots
    \end{bmatrix}
\end{align}is an isometry. That it is a lift of $T$ is clear from its representation. Furthermore, the isometric lift above is \textit{minimal}, i.e., $\cH$ is the smallest invariant subspace for $V$ that contains $\cH$. This is well-known -- see e.g., \cite[Chapter I]{MR2760647}.
\medskip

\noindent
\begin{proof}[The second proof]
    The idea comes from the first proof and the tools come from Proposition \ref{P:BabyStep}. Indeed, let
    $$
    Y=\begin{bmatrix}
        X&0\\
        \textbf{ev}_0^* D_X & M_z
    \end{bmatrix}:
    \begin{bmatrix}
        \cH\\
        H^2(\cD_X)
    \end{bmatrix}\to \begin{bmatrix}
        \cH\\
        H^2(\cD_X)
    \end{bmatrix}
    $$be the isometric lift of $X$ constructed by Sch\"affer \cite{MR68740}. Let us recall the well-known result that any two minimal isometric lifts of a contraction are unitarily equivalent (see e.g., \cite{BallSau2026}) and if $(X,T)$ satisfies a covariance relation, then the same relation is satisfied by any pair that is jointly unitarily equivalent to $(X,T)$. Thus it is enough to prove the proposition for this choice of isometric lift. Set
    $$
    S=\begin{bmatrix}
        T&0\\
        0&V
    \end{bmatrix}:
    \begin{bmatrix}
        \cH\\
        H^2(\cD_X)
    \end{bmatrix}\to \begin{bmatrix}
        \cH\\
        H^2(\cD_X)
    \end{bmatrix}
    $$where $V$ is the isometry constructed out of $X,T$ in Proposition \ref{P:BabyStep}. Since $T$ is an isometry, so is $S$. A simple matrix computation reveals that for the intertwining  
    $SY=Yf(S)$ to happen it is necessary and sufficient that
    $$
    TX=Xf(T),\quad V\textbf{ev}_0^*D_X=\textbf{ev}_0^*D_Xf(T),\quad\mbox{and}\quad VM_z=M_zf(V).
    $$
The first intertwining is part of the original assumption, the last intertwining is established in Proposition \ref{P:BabyStep}. For the intertwining in the middle, let us recall that $V=\operatorname{diag}(f^{\circ n}(C))_{n=0}^\infty$ and so it is equivalent to $CD_X=D_Xf(T)$, which is the very definition of $C$ as in \eqref{C}.    
\end{proof}
A quick corollary to the second proof is given below.
\begin{corollary}\label{C:addition}
    With the notations as in the statement of Proposition \ref{P:FinalNail}, if $T$ is unitary, then $S$ can be chosen to be unitary as well.
\end{corollary}
\begin{proof}
    We observed in the second proof of Proposition \ref{P:FinalNail} that the (22) entry of $S$ with respect to the decomposition $\cBm{\cH \\ \cK\ominus \cH}$ is given by
    $$
    V=\operatorname{diag}( f^{\circ n}(C))_{n=0}^\infty: \ell^2(\cD_X)\to\ell^2(\cD_X)
    $$where $C:\cD_X\to\cD_X$ is defined on $\operatorname{Ran}D_X$ by
    $$
    CD_Xh=D_Xf(T)h.
    $$In case $T$ is unitary, $f(T)$ is unitary too and hence so is $C$. This makes each diagonal entry of $V$ a unitary operator and therefore $V$ is a unitary. Since the (11) entry of $S$ is $T$, which by assumption is a unitary, $S$ is unitary as well.
\end{proof}
We record another important consequence.
\begin{corollary}\label{C:Uni/Iso/Ext}
    If $(U,V)$ is a unitary/isometry pair on $\cH$ and $f$ is an inner function such that
    $$
    f(U)V=VU.
    $$Then there is a unitary pair $(U_1,U_2)$ which extends $(U,V)$ such that
    \begin{align}\label{Aux4}
    f(U_1)U_2=U_2U_1.
    \end{align}Moreover, $U_2$ can be taken to be the minimal unitary extension of $V$.
\end{corollary}
\begin{proof}
 The adjoint of the given intertwining relation is $U^*V^*=V^*\widetilde{f}(U^*)$ where, as usual, $\widetilde{f}$ is the reflection of $f$. Let $U_2$ be the minimal unitary extension of $V$. Then $U_2^*$ then is a minimal isometric lift of $V^*$. By Corollary \ref{C:addition}, there is a unitary operator $U_1^*$ that extends $U^*$ such that $U_1^*U_2^*=U_2^*\widetilde{f}(U_1^*)$, which is equivalent to the desired intertwining \eqref{Aux4}. 
\end{proof}
One may ask whether Corollary \ref{C:Uni/Iso/Ext} extends to right inner covariance relations. The answer is affirmative; however, as the proof requires a technical analytic machinery developed later, we defer it to the next section. Continuing with our present development, we observe that
\begin{corollary}
    If $U$ is a unitary operator for which there is an isometry $V$ and an inner function $f$ such that $f(U)V=VU$, then there is a unitary extension $U_1$ of $U$ such that $f(U_1)$ is unitarily equivalent to $U_1$.
\end{corollary}
\begin{proof}
    This follows once $U_2^*$ is multiplied on the right side of \eqref{Aux4}.
\end{proof}
We now prove the main result of this section.
\begin{theorem}\label{T:Lift1}
Let $f$ be an inner function, $(T_1,T_2)$ on $\cH$ be a contractive operator pair such that
$
T_1T_2=T_2f(T_1).
$ Then there is a Hilbert space $\cK\supset\cH$, an isometric lift $(V_1,V_2)$ acting on $\cK$ of $(T_1,T_2)$ such that
\begin{align}\label{lift}
V_1V_2=V_2f(V_1).    
\end{align}
\end{theorem}
\begin{proof}
By item (1) of Theorem \ref{T:LiftEndo}, there is a Hilbert space $\cK'\supset\cH$, an isometric lift $V_1'$ of $T_1$ and a bounded operator $Y$ on $\cK'$ such that
    $$
    V_1'Y=Yf(V_1'),\quad \|Y\|=\|T_2\|\quad\mbox{and}\quad Y^*|_\cH=T_2^*
    $$Now let $V_2$ acting on $\cK$ be minimal isometric lift of $Y$. By Proposition \ref{P:FinalNail}, there is a Hilbert space $\cK\supset \cK'\supset\cH$, an isometric operator $V_1$ on $\cK$ such that $\cK_1'$ is reducing under $V_1$ and $(V_1,V_2)$ satisfies \eqref{lift}. It is clear from the construction that $(V_1,V_2)$ is a lift of $(T_1,T_2)$.
\end{proof}

\section{Extension and dilation to a unitary pair}
The main goal of this section is to prove that every contractive pair subject to a right inner covariance relation admits a unitary dilation that preserves the relation. This is achieved by first extending an isometric pair to a unitary pair while maintaining the covariance structure, which we execute in two steps.
\begin{proposition}\label{P:IsoExt}
    Suppose $f$ is an inner function, $(V_1,V_2)$ is an isometric pair such that
    \begin{align}\label{IsoExt}
    V_1V_2=V_2f(V_1).
    \end{align}
    Then there is a Hilbert space $\cK\supset \cH$, a unitary/isometry pair $(U_1,\widetilde V_2)$ that extends $(V_1,V_2)$ and satisfies 
    \begin{align}\label{IsoExt''}
    U_1\widetilde V_2=\widetilde V_2f(U_1).
    \end{align}
\end{proposition}
\begin{proof}
The adjoint of \eqref{IsoExt} is $\widetilde{f}(V_1^*)V_2^*=V_2^*V_1^*$, where $\widetilde{f}(z)=\overline{f(\overline{z})}$. Let 
    \begin{align}\label{U1prime}
    U_1'=\begin{bmatrix}
        V_1 & *\\
        0 & *
    \end{bmatrix}:\begin{bmatrix}
        \cH \\
        \cK_1\ominus \cH
    \end{bmatrix}\to \begin{bmatrix}
        \cH \\
        \cK_1\ominus \cH
    \end{bmatrix}
    \end{align} on $\cK_1$ be a unitary extension of $V_1$. Then $U_1'^*$ is an isometric lift of $V_2^*$. Apply Theorem \ref{T:Classical_Iso_ILT} to the choice $(T,X)=(V_1^*,V_2^*)$ in order to obtain a bounded operator $Y^*$ such that
    \begin{align}\label{Aux3}
    \widetilde{f}(U_1'^*)Y^*=Y^*U_1'^*, \quad \|Y\|=1\quad\mbox{and}\quad Y=\begin{bmatrix}
        V_2 & *\\
        0 & *
    \end{bmatrix}.
    \end{align}The adjoint of the intertwining relation above is $U_1'Y=Yf(U_1')$. If 
    \begin{align}\label{wideTilde}
    \widetilde V_2=\begin{bmatrix}
        Y & 0\\
        * & *
    \end{bmatrix}: \begin{bmatrix}
        \cK_1 \\
        \cK\ominus \cK_1
    \end{bmatrix}\to \begin{bmatrix}
        \cK_1 \\
        \cK\ominus \cK_1
    \end{bmatrix}
    \end{align} acting on $\cK$ is a minimal isometric lift of $Y$, then by Proposition \ref{P:FinalNail} and Corollary \ref{C:addition}, there is a unitary operator $U_1$ on $\cK$ such that
    \begin{align}\label{U1}
    U_1\widetilde V_2=\widetilde V_2f(U_1) \quad\mbox{and}\quad U_1=\begin{bmatrix}
        U_1' & 0\\
        0 & *
    \end{bmatrix}:\begin{bmatrix}
        \cK_1 \\
        \cK\ominus \cK_1
    \end{bmatrix}\to \begin{bmatrix}
        \cK_1 \\
        \cK\ominus \cK_1
    \end{bmatrix}.
    \end{align} Thus we have obtained the promised intertwining \eqref{IsoExt''}. It is clear from the matrix representation of $U_1$ in \eqref{U1} and the matrix representation of $U_1'$ in \eqref{U1prime} that $U_1$ extends $V_1$.  To see that the isometry $\widetilde{V_2}$ also extends $V_2$, let us note from the matrix representations in \eqref{Aux3} and \eqref{wideTilde} that $\widetilde{V_2}$ has the matrix structure given by
    \begin{align*}
     \widetilde V_2=\begin{bmatrix}
        V_2 & * & 0\\
        0& * & 0\\
        A_{31} & A_{32} & A_{33}
    \end{bmatrix}: \begin{bmatrix}
        \cH \\
        \cK_1\ominus \cH\\
        \cK\ominus \cK_1
    \end{bmatrix}\to  \begin{bmatrix}
        \cH \\
        \cK_1\ominus \cH\\
        \cK\ominus \cK_1
    \end{bmatrix}.
    \end{align*}Since $V_2$ is already an isometry, the only operator $A_{31}$ so that $\widetilde{V_2}$ can be a contraction is the zero operator.
\end{proof}

To obtain a full unitary dilation theorem, we need to extend a unitary/isometry pair $(U,V)$ satisfying an inner covariance relation to a unitary pair satisfying the same relation. We use the Borel functional calculus of a unitary operator -- see e.g., \cite[Chapter IX]{MR768926}.

We also need a Borel right inverse (also known as a Borel section) of a continuous function. A continuous surjective map between two compact spaces may not have a continuous section. For an example, the map $f:\mathbb T\to \mathbb T$ defined by $f(z)=z^n$ for some $n\ge 2$ does not have a continuous section. However, as the following theorem ensures, we can always get a Borel section (see \cite[Theorem 4.2]{MR226684}, or \cite[Theorem I.16]{MR1775825}).
\begin{theorem}(Federer and Morse)\label{T:Borelsection}
Let $X$ and $Y$ be two compact metric spaces and $f$ is a continuous map from $X$ onto $Y$. Then there is a Borel set $B\subseteq X$ such that $f$ restricted to $B$ is one-to-one and onto. Moreover, the inverse of the restriction is a Borel section of $f$.
\end{theorem} We shall use this in what follows.
\begin{proposition}\label{P:LongWait}
    Let $(U,V)$ be a unitary/isometry pair acting on $\cH$ such that 
    \begin{align}\label{Aux5}
    UV=Vf(U)    
    \end{align}
    for some surjective Borel function $f$ on $\bT$, then there is a joint unitary extension $(U_1,U_2)$ of $(U,V)$ such that $$
    U_1U_2=U_2f(U_1).
    $$Moreover, $U_2$ can be chosen to be the minimal unitary extension of $V$.
\end{proposition}
\begin{proof}
    Since $U$ is unitary, $f(U)$ is also unitary and so the intertwining \eqref{Aux5} is the same as $U^*V=Vf(U)^*$. These two intertwining relations show that 
    $$
    UD_{V^*}=U(I-VV^*)=U-Vf(U)V^*=U-VV^*U=D_{V^*}U
    $$and so $\cD_{V^*}$ is reducing under $U$. Define
    $$
    W_0=U|_{\cD_{V^*}}=\int_\bT \zeta\;dE(\zeta),
    $$where $E$ is the spectral measure for the unitary operator $W_0$. By Theorem \ref{T:Borelsection}, there is a Borel section $g:\bT\to\bT$ of $f$. For every $k\geq 1$, define
    the unitary operators $W_k$ on $\cD_{V^*}$ by $$
    W_k=\int_\bT g^{\circ k}(\zeta)\;dE(\zeta).
    $$It then follows that
    \begin{align}\label{Aux6}
    UD_{V^*}=D_{V^*}W_0
    \end{align} and by the Borel functional calculus for unitary operators, we have
    \begin{align}\label{Aux7}
    f(W_k)=W_{k-1} \quad\mbox{for every }k\geq 1.
    \end{align} Set unitary operators $(U_1,U_2)$ on $\cK=\cH\oplus \cD_{V^*}\oplus \cD_{V^*}\oplus \cdots$ as follows:
    \begin{align*}
        U_1 = \begin{bmatrix}
            U&&  &  & \\
             &  W_1&   &  &\\
             &  &  W_2 &  &  \\
             &  &  &  \ddots &  \\        
        \end{bmatrix}\quad\mbox{and}\quad
        U_2=\begin{bmatrix}
            V&D_{V^*}&  &  & \\
             &  &  I_{\cD_{V^*}} &  &\\
             &  &  &  I_{\cD_{V^*}} &  \\
             &  &  &  & \ddots \\        
        \end{bmatrix},
    \end{align*}where all the unspecified entries contain the zero operator. It is clear that $U_2^*$ is the minimal isometric lift of $V^*$ and therefore $U_2$ is the minimal unitary extension of $V$. We compute
    $$
    U_1U_2=\begin{bmatrix}
            UV&UD_{V^*}&  &  & \\
             &  & W_1 &  &\\
             &  &  &  W_2 &  \\
             &  &  &  &  \ddots\\        
        \end{bmatrix}
    $$whereas
    $$
    U_2f(U_1)=
    \begin{bmatrix}
        Vf(U) & D_{V^*}f(W_1) & & \\
        & & f(W_2) & \\
        & & & f(W_3) \\
        & & & & \ddots
    \end{bmatrix}
    $$What remains is to see from the above matrix computation that $U_1U_2=U_2f(U_1)$ by the relations \eqref{Aux6} and \eqref{Aux7}. 
\end{proof}
Every finite Blaschke factor satisfies all the hypothesis imposed on the function $f$ as in Proposition \ref{P:LongWait}. As a consequence of Propositions \ref{P:IsoExt} and \ref{P:LongWait} we get the following result.
\begin{theorem}\label{T:UniExtIso}
    If $(V_1,V_2)$ is an isometric pair such that 
    $$
    V_1V_2=V_2f(V_1),
    $$for some inner function $f$, then there is a unitary extension $(U_1,U_2)$ of $(V_1,V_2)$ that satisfies the same relation.
\end{theorem}
Theorems \ref{T:UniExtIso} and \ref{T:Lift1} together imply the following unitary dilation result. 
\begin{theorem}\label{T:UniDil}
Suppose $(T_1,T_2)$ is a contractive pair acting on $\cH$ satisfying the covariance relation $T_1T_2=T_2f(T_1)$ for a Blaschke function $f$. Then there is a Hilbert space $\cK\supset\cH$, a unitary pair $(U_1,U_2)$ on $\cK$ such that 
\begin{enumerate}
    \item $\cH$ is semi-invariant under $(U_1,U_2)$,
    \item $P_\cH U_j|_\cH=T_j$ for $j=1,2$, 
    \item $U_1U_2=U_2f(U_1)$.
\end{enumerate}
\end{theorem}

\begin{remark}
Here we remark that analogues of Theorems \ref{T:Lift1} and \ref{T:UniDil} hold true for inner left covariance relation as well. Suppose that $(T_1,T_2)$ is a contractive pair such that $f(T_1)T_2=T_2T_1$. The adjoint relation is $T_1^*T_2^*=T_2^*\widetilde{f}(T_1^*)$ where $\widetilde{f}$ is the reflection of $f$. By Theorem \ref{T:UniDil}, there is a Hilbert space $\cK$, a unitary operator pair $(U_1,U_2)$ such that 
\begin{enumerate}
    \item $\cH$ is semi-invariant under $(U_1,U_2)$, 
    \item $P_\cH(U_1,U_2)|_\cH=(T_1,T_2)$ and 
    \item $U_1^*U_2^*=U_2^*\widetilde{f}(U_1^*)$ or equivalently, $f(U_1)U_2=U_2U_1$.
\end{enumerate}Semi-invariance means that there are closed orthogonal subspaces $\cM,\cN$ of $\cK$ such that $\cK=\cM\oplus\cH\oplus\cN$, the spaces $\cM$ and $\cK_+:=\cM\oplus\cH$ are $(U_1,U_2)$-invariant, and with the isometric pair 
$$
(V_1,V_2)=(U_1,U_2)|_{\cK_+},
$$we have $V_j^*|_\cH=T_j^*$ for each $j=1,2$. It is trivial to check from the given information that $(V_1,V_2)$ also satisfies the left inner covariance relation $f(V_1)V_2=V_2V_1$.
\end{remark}

\section{The non-commutative polydisk as a complete spectral set}
In the classical setting, the matrix-valued von Neumann inequality generally fails for commuting contractive $d$-tuples when $d \ge 3$ -- see \cite{MR268710, MR355642}. However, it is not difficult to see that if a commuting contractive tuple $\bm T$ lifts to a commuting isometric tuple $\bm V$ annihilated by a polynomial $p$, the inequality readily holds for $\bm T$ over the algebraic variety $\mathcal Z(p)\cap\bD^d$. In this section we show that non-commutative function theory presents a striking dichotomy: while the matrix-valued von Neumann inequality holds universally for contractive tuples of arbitrary finite length over the free polydisk $\bD^{d,\rm NC}$ as defined below in \eqref{FreeAnalysis}, it fails over the non-commutative inner variety $\cV_f^{\rm NC}$ as defined below in \eqref{InnerVariety} -- even when the underlying operator pair admits a covariant isometric lifting.

A background in non-commutative function theory is not strictly required to understand the results presented here. We omit these preliminary concepts, but direct interested readers to \cite{MR4411370, MR4346386, MR4814008} and the references therein.

Consider the non-commutative (NC) polydisk 
\begin{align}\label{FreeAnalysis}
\overline{\bD}^{d, \rm NC}: =\bigcup_{n=1}^\infty \left\{(X_1,X_2,\dots,X_d)\in M_n(\bC)^d:\|X_j\|\leq 1\mbox{ for }j=1,2,\dots,d\right\}.
\end{align} The following theorem shows that $\overline{\bD}^{d, \rm NC}$ is a complete spectral set, i.e., for every matrix-valued NC polynomial
$$
P=\begin{bmatrix}
    p_{ij}
\end{bmatrix}_{i,j=1}^N,
$$where $p_{ij}\in \bC\langle z_1,z_2,\cdots,z_d\rangle$ are scalar-valued NC polynomials, we have
\begin{align}\label{CompleteIneq}
\|P(T_1,T_2,\cdots,T_d)\|_{\cB(\cH\otimes \bC^N)} \leq \sup_{\bD^{\rm d, NC}} \|P(z_1,z_2,\dots,z_d)\|.
\end{align}

\begin{theorem}\label{T:CompleteSpec}
Every contractive operator tuple $(T_1,T_2,\dots,T_d)$ has $\overline{\bD}^{d, \rm NC}$ as a complete spectral set.
\end{theorem}
\begin{proof}
Pick a vector $v$ in $\cH$ of norm $1$. Then there is a countable set of orthonormal vectors $\{e_j:j\geq1\}$ so that $v=\sum c_j e_j$. Let $\cH_n=\bigvee\{e_1,e_2,\dots,e_n\}$ and $v_n=\sum_{j=1}^nc_je_j$. Now consider the matrix tuple
$$
(X_{1,n},X_{2,n},\cdots ,X_{d,n}) = P_{\cH_n}(T_1,T_2,\cdots,T_d)|_{\cH_n}.
$$We show that $P(X_{1,n},X_{2,n},\cdots,X_{d,n})v_n\to P(T_1,T_2,\cdots,T_d)v$. Note that once this limit is established, the inequality will follow. The limit will be established in two steps:
\medskip

\noindent
{\sf Step 1:} For each $j=1,2,\cdots,d$ and $\ell \geq 1$, $(X_{j,n})^\ell \to T_j^\ell$ as $n \to\infty$.

Let us note that $\|X_{j,n}\|\leq 1$ for each $j=1,2$ and $n\geq1$. Therefore if $\{h_n\}\subset\cH_n$ is any sequence of vectors in $\cH$ such that $h_n\to h$, then $X_{j,n}h_n \to T_jh$. This can be seen from the following computation:
\begin{align*}
    \|X_{j,n}h_n - T_jh\|& \leq \|X_{j,n}h_n - P_{\cH_n}T_jh \| + \|P_{\cH_n}T_jh  -  T_jh\|\\& \leq \|P_{\cH_n}T_j(h_n - h)\| +\|P_{\cH_n} (T_jh)-T_jh\|.
\end{align*}
Proceeding by induction we get the desired result of step $1$.
\medskip

\noindent
{\sf Step 2:} It is easy to see that $X_{j,n}X_{k,n}h_n\to T_jT_kh$ for all $j,k\in \{1,2,\cdots,d\}$. Now for any word $w(X_{1,n},X_{2,n},\cdots ,X_{d,n})=Z_1Z_2\cdots Z_m$ of length $m$, where $Z_i\in \{X_{1,n},X_{2,n},\cdots ,X_{d,n}\}$ we use induction on $m$ to show that $$w(X_{1,n},X_{2,n},\cdots ,X_{d,n})v_n\to (T_1,T_2,\cdots,T_d)v.$$
This completes the proof.
\end{proof}
We show that despite the dilation Theorem \ref{T:UniDil} and the NC polydisk $\overline{\bD}^{d,\rm NC}$ being a complete spectral set for every contractive tuple, it is not true for a contractive pair $(T_1,T_2)$ satisfying an inner covariance relation $T_1T_2=T_2f(T_1)$ to have the NC inner variety $\cV_f^{\rm NC}$ defined as
\begin{align}\label{InnerVariety}
=\bigcup_{n=1}^\infty\{(X_1,X_2)\in M_n(\bC)^2:X_1X_2=X_2f(X_1)\mbox{ and } \|X_j\|\leq 1\mbox{ for }j=1,2\}
\end{align}as a complete spectral set. We demonstrate this fact by a concrete example below.
\begin{example}\label{E:NotComplete}
We take the inner function $f$ to be the irrational rotation $f(z)=e^{2\pi i\theta}z$ where $\theta\in\bR\setminus \mathbb Q$.
Suppose $(X_1,X_2)$ is an $N\times N$ matrix pair satisfying $X_1X_2=e^{2\pi i\theta}X_2X_1$. We first prove that
\medskip

\noindent
{\sf Claim:} $X=X_1X_2$ must be nilpotent. 

To see this, note that
$$
XX_2=X_1X_2X_2=qX_2X_1X_2=qX_2X
$$and so by an inductive argument, $X^\ell X_2=q^{\ell}X_2X^\ell$ for every $\ell \geq 1$. We use this to note that for every $\ell \geq 1$,
$$
\operatorname{Tr} X^\ell = \operatorname{Tr}(X^{\ell-1}X_1X_2) =
\operatorname{Tr}(X_2X^{\ell -1}X_1) = 
\overline{q}^{\ell -1}\operatorname{Tr}(X^{\ell -1}X_2X_1)
= \overline{q}^\ell \operatorname{Tr}X^\ell.
$$Since $\theta$ is irrational, we must have 
$$
\operatorname{Tr}(X^\ell) = 0 \quad\mbox{for every }\ell \geq 1.
$$This shows that $X$ is nilpotent.

We show that any $q$-commuting unitary pair $(U_1,U_2)$ such that the unitary operator $U=U_1U_2$ has $-1$ in its spectrum cannot have $\cV_f^{\rm NC}$ as a complete spectral set. Examples of such unitary pairs include $(R_q, M_\zeta)$ on $L^2(\bT)$ or $(M_{\zeta_1}R_q^{(2)},R_{\overline q}^{(2)}M_{\zeta_2})$ on $L^2(\bT\times \bT)$ where $R_q^{(2)}$ is the rotation operator defined by
$$
R_q^{(2)}: \zeta_1^m\zeta_2^n \mapsto q^{\min\{m,n\}}\zeta_1^m\zeta_2^n \quad \mbox{ for every integers }m,n.
$$
Consider the polynomial
$$
P(z_1,z_2) = \begin{bmatrix}
    z_1z_2 & 0\\
    1-z_1z_2 & 0
\end{bmatrix}.
$$Then
$$
\|P(U_1,U_2)\|^2 = \|P(U_1,U_2)^*P(U_1,U_2)\| =  \|I + (I-U)^*(I-U)\| = 
\|3I - U - U^*\|.
$$The full spectrum assumption $\sigma(U) =\bT$ together with the Gelfand theory gives
$$
\|P(U_1,U_2)\|^2 = \sup_{\zeta\in\bT} (3 -\zeta - \overline{\zeta}) = 5.
$$ On the other hand, for any $N\times N$ matrix pair $(X_1,X_2)$ such that $X=X_1X_2$ is nilpotent, we have
$$
\|P(X_1,X_2)\|^2 = \|X^*X\| + \|(I-X)^*(I-X)\| \leq 1 + \|I-X\|^2.
$$We now prove that $\|I-X\|<2$ for any nilpotent matrix $X$. Once we show this, we can conclude that the chosen unitary pair $(U_1,U_2)$ does not have $\cV_{f}^{\rm NC}$ as a complete spectral set. It is known (see \cite{MR1072339}) that if $X$ is a nilpotent contraction and if $\ell$ is the least positive integer such that $X^\ell = 0$, then the numerical radius of $X$, $w(X)$ must satisfy
$
w(X) \leq \cos\left( \pi/(\ell + 1)\right).
$
Since $\cos\left(\pi/(\ell + 1)\right)<1$ for every $\ell\geq 1$, the claim is established.

\end{example}

\section{Dynamic interpolation/extension theorem}
Let $f$ be a holomorphic self-map of $\bD$. A point $z\in \bD$ is called a \textit{periodic point} of $f$ if there exists a positive integer $n_z$ such that $f^{\circ n_z}(z)=z$, and the least positive integer is called the \textit{period} of $z$. If there is no such positive integer, we use the convention that $n_z=\infty$. In the theorem below, $\bC^\infty = \ell^2$, the Hilbert space of square summable sequences.
\begin{theorem}\label{T:BoundExt}
Let $\bCS\subset \bD$ and $g:\bCS\to\bC$ be a function for which there is a Blaschke function $f$ such that $\bCS\supset \{\mathcal{O}_f(z):z\in \bCS\}$ and 
$$
g(f(z)) = g(z) \quad\mbox{for all }z\in\bCS.
$$ Then $g$ has a bounded holomorphic extension on $\bD$ if and only if there exists $M>0$ such that the matrix-valued function $\Delta:(z,w)\to \cB(\bC^{n_z},\bC^{n_w})$ on $\bCS\times \bCS$ defined by
$$
\Delta(z,w) = \begin{bmatrix}
    \dfrac{M^2\alpha^2}{1-f^{\circ j}(z)\overline{f^{\circ l}(w)}}-\dfrac{g(z)\overline{g(w)}}{1-f^{\circ j+1}(z)\overline{f^{\circ l+1}(w)}}
\end{bmatrix}_{j,l=(0,0)}^{(n_z-1, n_w-1)}
$$ is a positive semi-definite kernel where 
$\alpha^2=(1+|f(0)|)(1-|f(0)|)^{-1}.$ 
\end{theorem}

\begin{proof}
Let us first assume that a bounded holomorphic extension of $g$ exists i.e., there exists $G\in H^{\infty}(\bD)$ such that $\|G\|\le M$ and $G|_{\mathcal S}=g$. Consider the subspace $$\bCG=\bigvee \{k_{f^{\circ{n}}(z)}:n\ge 0 \text{ and } z\in \mathcal{S}\}\subset H^2.$$ Define an operator $D$ on $\bCG$ by $D=P_{\bCG}M_GC_f\mid_{\bCG}$. Then the operator $D^*:\bCG\to\bCG$ is given by 
$$D^*(k_{f^{\circ n}}(z))=\overline{g(z)}k_{f^{\circ n+1}(z)}.$$ It is well known that 
$
\|C_f\|_{H^2}\leq (1+|f(0)|)(1-|f(0)|)^{-1}
$ and $\|M_G\|=\|G\|_{\infty}$. Thus if we denote $M:=\|G\|_\infty$, then
$$
\|D\|^2\le \|G\|_{\infty}^2(1+|f(0)|)(1-|f(0)|)^{-1}= M^2\alpha^2.
$$Hence for any $x=\sum_{i=1}^{n}\sum_{j=0}^{m}\alpha_{ij}k_{f^{\circ j}(z_i)}\in \bCG$, we have 

\begin{align*}0&\le M^2\alpha^2\|x\|^2-\|D^*x\|^2\\&=M^2\alpha^2\|\sum_{i=1}^{n}\sum_{j=0}^{m}\alpha_{ij}k_{f^{\circ j}(z_i)}\|^2-\|\sum_{i=1}^{n}\sum_{j=0}^{m}\alpha_{ij}\overline{g(z_i)}k_{f^{\circ j+1}(z_i)}\|^2
\\&= \Bigg (M^2\alpha^2\sum_{i,k=1}^{n}\sum_{j,l=0}^{m}\alpha_{kl}\overline{\alpha_{ij}}\langle k_{f^{\circ l}(z_k)},k_{f^{\circ j}(z_i)}\rangle\\&-\sum_{i,k=1}^{n}\sum_{j,l=0}^{m}\alpha_{kl}\overline{\alpha_{ij}}\overline{g(z_k)} g(z_i)\langle k_{f^{\circ l+1}(z_k)},k_{f^{\circ j+1}(z_i)}\rangle \Bigg ) \\
&=\sum_{i,k=1}^{n}\sum_{j,l=0}^{m}\alpha_{kl}\left(\dfrac{M^2\alpha^2}{1-f^{\circ j}(z_i)\overline{f^{\circ l}(z_k)}}-\dfrac{g(z_i)\overline{g(z_k)}}{1-f^{\circ j+1}(z_i)\overline{f^{\circ l+1}(z_k)}}\right)\overline{\alpha_{ij}}\\&
=\sum_{i,k=1}^{n} \langle \Delta(z_i,z_k)\bm{\alpha}_k,\boldsymbol{\alpha}_i\rangle,
\end{align*}
where $\boldsymbol{\alpha}_i=(\overline{\alpha_{i0}},\overline{\alpha_{i1}},\cdots,\overline{\alpha_{im}})^T$ for all $i=1,2,\cdots, n$ and for all $j,l=0,1,\cdots,m$.

To prove the converse, let us define $D^*:\bCG\to\bCG$ by 
\begin{align}\label{DefineD}
D^*(k_{f^{\circ n}}(z))=\overline{g(z)}k_{f^{\circ n+1}(z)}.    
\end{align}
As observed in the forward direction, the positivity condition on $\Delta$ is precisely the same as continuity of $\Delta$. Let us define the operator $A^*=M_z^*|_\bCG$. Therefore we have $A^*(k_{f^{\circ n}(z)})=\overline{f^{\circ n}(z)}k_{f^{\circ n }(z)}$. It is easy to check that $f(A)D=DA$. The isometry $M_z$ on $H^2$ is clearly a lift of $A$ and so by the Lifting theorem (Theorem \ref{T:LiftEndo}), there exists a bounded operator $B$ on $H^2$ such that 
$$
f(M_z)B=M_fB=BM_z \quad\mbox{and}\quad D^*=B^*\mid_{\bCG}.
$$ By Theorem \ref{T:f-commutant of vector-shift} $B=M_GC_f$ for some $G\in H^{\infty}(\bD)$. This together with \eqref{DefineD} gives $G(z)=g(z)$ for all $z\in \mathcal{S}.$   
\end{proof}

We now illustrate Theorem \ref{T:BoundExt} by considering several special cases.

\begin{example}\label{Ex:Pick}
It is possible to deduce from Theorem \ref{T:BoundExt} the classical Pick interpolation theorem, which states that given a finite subset $\bCS=\{z_1,z_2,\dots,z_N\}$ of $\bD$, a function $g:\bCS\to\bC$ has a bounded holomorphic extension to $\bD$ if and only if there is $M>0$ such that
$$
P = \begin{bmatrix}
\frac{M^2-g(z_i)\overline{g(z_k)}}{1-z_i\overline{z_k}}
\end{bmatrix}\succeq 0.
$$To deduce this fact from Theorem \ref{T:BoundExt} we consider any subset $\bCS$ of $\bD$ and the Blaschke function $f(z)=z$. In this case, $\bCO_f(z)=\{z\}$ for every $z\in\bD$ and so $S$ contains all the orbits of its points under $f$ as is required in Theorem \ref{T:BoundExt}. Moreover, since $\alpha$ as in the statement of Theorem \ref{T:BoundExt} is $1$ in this case, the operator-valued function $\Delta$ as in the statement of the Theorem \ref{T:BoundExt} turns out to be the scalar-valued function $\Delta:\bCS\times\bCS\to\bC$ given by
 \begin{align*} 
\Delta(z,w)=\frac{M^2-g(z)\overline{g(w)}}{1-z\overline{w}}.
\end{align*} Thus a function $g:\bCS\to\bC$ holomorphic or not, has a bounded holomorphic extension to $\bD$ if and only if the function $\Delta$ as above is a positive semi definite kernel. Note that $\Delta$ is exactly the same as the Pick matrix $P$ as above when $\bCS=\{z_1,z_2,\dots,z_N\}$. 
\end{example}

\begin{example}
For this example, let us take $\bCS$ to be a circular disk $\bD_r=\{z:|z|<r<1\}$ and the Blaschke function $f$ to be a rotation map $f(z)=qz$, $q\in\bT$. Then $\bCS$ contains all the orbits of its points under $f$. Theorem \ref{T:BoundExt} answers when a rotation invariant function $g$ on $\bCS$ (holomorphic or not) has a bounded holomorphic extension to the whole $\bD$. The answer depends on whether $q$ comes from a rational or an irrational point in $\bR$. Indeed, suppose $q$ is a primitive $N$-th root of unity. In this case, \textit{a function $g:\bD\to\bC$ with the rotation invariance property $g(qz)=g(z)$ has a bounded holomorphic extension to $\bD$ if and only if there exists $M>0$ such that the $N\times N$ matrix-valued function $\Delta:\bCS\times\bCS\to M_N(\bC)$ given by
\begin{align}\label{qPick}
\Delta(z,w) = \begin{bmatrix}
    \dfrac{M^2-g(z)\overline{g(w)}}{1-q^{l-j}z\overline w}
\end{bmatrix}_{j,l=1}^{N}    
\end{align}
is positive semi-definite.} It is interesting to note that, when $q$ is an $N$-th root of $1$, given any function $g:\bS\to\bC$ one can construct a $q$-invariant function $\hat g$ on $\bCS$ with $\|\hat g\|_\infty\leq \|g\|_\infty$ as follows:
\begin{align*}
    \hat g(z) = \frac{g(z)+g(qz)+\cdots+ g(q^{N-1}z)}{N}.
\end{align*}However, this observation is of limited utility, as an extension of $\hat g$ is not necessarily an extension of $g$. Consequently, the unimodular number $q$ cannot generally be omitted from the matrix \eqref{qPick}.

When $q=e^{2\pi i \theta}$ for an irrational $\theta$, the same conclusion as in the rational case above holds except that the matrix as in \eqref{qPick} will be an infinite matrix (i.e., an operator on $\ell^2$) because the orbit length $n_z$ for every $z\in \bCS$ is infinite.
\end{example}

We now present an application of Theorem \ref{T:f-commutant of vector-shift}.
\begin{theorem}\label{T:Bdd-holomorphic}
A given function $g:\bD\to\bC$ is bounded and holomorphic on $\bD$ if and only if there exists $M>0$ and a Blaschke function $f$ such that the function $\Delta_f:\bD\times \bD \to \bC$ defined by
$$
\Delta(z,w) =M^2\alpha^2k(z,w)-g(z)k(f(z),f(w))\overline{g(w)}
$$is a positive semi-definite kernel where $\alpha^2 = (1+|f(0)|)(1-|f(0)|)^{-1}$.
\end{theorem}
 \begin{proof}
Let us first assume that $g$ is bounded holomorphic. Then for $f(z)=z$, the conclusion follows from the discussion in Example \ref{Ex:Pick}.  

To prove the converse, let us define a linear transformation $D^*$ on finite linear combinations of kernel functions by linearly extending the assignment:
\begin{align}\label{DefineD*}
D^*k_z=\overline{g(z)}k_{f(z)} \quad \mbox{for all }z\in \bD.
\end{align}
Then for $M>0$ and $x\in H^2$ with $x=\sum_{i=1}^{m}\alpha_{i}k_{z_i}$, we have 
\begin{align*}
 M^2\alpha^2\|x\|^2-\|D^*x\|^2
 =M^2\alpha^2\|\sum_{i=1}^{m}\alpha_{i}k_{z_i}\|^2-\|\sum_{i=1}^{m}\alpha_{i}\overline{g(z_i)}k_{f(z_i)}\|^2,
\end{align*}which, when written in terms of inner products, gives
\begin{align*}
&M^2\alpha^2\sum_{i,j=1}^{m}\alpha_{j}\overline{\alpha_{i}}\langle k_{z_j},k_{z_i}\rangle-\sum_{i,j=1}^{m}\alpha_{j}\overline{\alpha_{i}}g(z_i)\overline{g(z_j)}\langle k_{f(z_j)},k_{f (z_i)}\rangle\\
&=\sum_{i,j=1}^{m}\alpha_{j}\left(\dfrac{M^2\alpha^2}{1-z_i\overline{z_j}}-\dfrac{g(z_i)\overline{g(z_j)}}{1-f(z_i)\overline{f(z_j)}}\right)\overline{\alpha_{i}}.
\end{align*}
The positivity assumption on $\Delta$ makes $D$ a bounded operator on $H^2$ when extended the above assignment linearly and continuously. It is easy to check that we have the left covariance relation 
$$
f(M_z)D=DM_z.
$$ By Theorem \ref{T:f-commutant of vector-shift}, we must have $D=M_GC_f$ for some $G\in H^{\infty}(\bD)$ giving the desired result because $G(z)=g(z)$ for all $z\in \bD$.
\end{proof}

\medskip

\textbf{Acknowledgment.} This project was conceived while the first author visited the Indian Statistical Institute, Bangalore, and the second author visited the Indian Institute of Science, Bangalore, during the summer of 2026. The authors gratefully acknowledge the support and hospitality extended by both institutions.

\bibliographystyle{plain}
\bibliography{ref}
\end{document}